\documentclass{amsart}
\usepackage{amsfonts,amsmath,color,graphicx,calligra,mathrsfs,tikz,tikz-cd,mathtools}
\usepackage{amsthm,amssymb}

\usetikzlibrary{arrows}

\usepackage[english]{babel}
\usepackage{comment}
\usepackage{hyperref}

\newtheorem{thm}{Theorem}[section]

\newtheorem{prop}[thm]{Proposition}
\theoremstyle{definition}
\newtheorem{defn}[thm]{Definition}
\theoremstyle{remark}

\theoremstyle{definition}
\newtheorem{ex}[thm]{Example}
\theoremstyle{definition}

\theoremstyle{remark}

\numberwithin{equation}{section}

\newcommand{\LLL}{\mathbb{L}}
\newcommand{\N}{\mathbb{N}}

\newcommand{\Z}{\mathbb{Z}}

\newcommand{\HH}{\mathrm{H}}
\newcommand{\aaa}{\mathfrak{a}}
\newcommand{\bbb}{\mathfrak{b}}

\newcommand{\coker}{\operatorname{coker}}

\newcommand{\EExt}{\mathscr{E}\text{\kern -3pt {\calligra\large xt}}\ }

\newcommand{\HHom}{\mathscr{H}\text{\kern -3pt {\calligra\Large om}}\, }
\newcommand{\HHH}{\operatorname{H}\hspace{-0.1em}\operatorname{H}}

\newcommand{\im}{\operatorname{Im}}

\newcommand{\LS}{\operatorname{LS}}

\newcommand{\Tor}{\operatorname{Tor}}

\newcommand{\gp}{\text{gp}}

\newcommand{\arrowr}{\arrow{r}}
\newcommand{\arrowd}{\arrow{d}}
\newcommand{\limdir}{\varinjlim}

\newcommand{\Lim}[1]{\raisebox{0.7ex}{\scalebox{0.8}{$\displaystyle\limdir_{i}\;$}}}

\newcommand*{\longhookrightarrow}{\ensuremath{\lhook\joinrel\relbar\joinrel\rightarrow}}

\makeatletter
\def\@tocline#1#2#3#4#5#6#7{\relax
  \ifnum #1>\c@tocdepth 
  \else
    \par \addpenalty\@secpenalty\addvspace{#2}%
    \begingroup \hyphenpenalty\@M
    \@ifempty{#4}{%
      \@tempdima\csname r@tocindent\number#1\endcsname\relax
    }{%
      \@tempdima#4\relax
    }%
    \parindent\z@ \leftskip#3\relax \advance\leftskip\@tempdima\relax
    \rightskip\@pnumwidth plus4em \parfillskip-\@pnumwidth
    #5\leavevmode\hskip-\@tempdima
      \ifcase #1
       \or\or \hskip 1em \or \hskip 2em \else \hskip 3em \fi%
      #6\nobreak\relax
    \hfill\hbox to\@pnumwidth{\@tocpagenum{#7}}\par
    \nobreak
    \endgroup
  \fi}
\makeatother

\begin{document}

\title[Two results on the logarithmic cotangent complex]{Two results on the logarithmic cotangent complex}%
\author{Jesús Conde-Lago and Javier Majadas}%
\address{Departamento de Matem\'aticas, Facultad de Matem\'aticas, Universidad de Santiago de Compostela, E15782 Santiago de Compostela, Spain}%
\email{jesus.conde@usc.es, j.majadas@usc.es}%

\thanks{This work was partially supported by Agencia Estatal de Investigaci\'on (Spain), grant PID2020-115155GB-I00 (European FEDER
support included, UE) and by Xunta de Galicia through the Competitive Reference Groups (GRC), ED431C 2019/10.}


\maketitle

\begin{abstract}
  In this paper we give two results on the logarithmic cotangent complex: we construct logarithmic analogues to the complex of Lichtenbaum and Schlessinger and to Quillen's fundamental spectral sequence.
\end{abstract}

Andr\'e-Quillen homology first appeared in dimension 0 and 1 implicitly in \cite{EGA4}. Subsequently, Lichtenbaum and Schlessinger in \cite{LS} introduce the second homology module as the homology of a simple complex. The general definition was then given by Andr\'e \cite{An-MS} and Quillen \cite{Quillen-MIT} using simplicial methods. The fact that the second Andr\'e-Quillen homology module agrees with that of Lichtenbaum and Schlessinger was proved in \cite[15.12]{An-1974}. Due to its simplicity, the complex of Lichtenbaum and Schlessinger is still useful if we do not need higher homology modules. It is for example the only one used in the  entire book \cite{MR}.

A logarithmic version of the cotangent complex was introduced by Gabber (see \cite{Ol}). Since the second homology module of the logarithmic cotangent complex detects Kato's log regularity \cite{CM}, a logarithmic analogue of the complex of Lichtenbaum and Schlessinger would be desirable. This is the first purpose of this paper. The second one is to give a logarithmic analogue of Quillen's fundamental spectral sequence \cite[Theorem 6.8]{Quillen-MIT}
	\[ E_{p,q}^2 = H_{p+q}(S^q\LLL_{B|C}) \Rightarrow \Tor_{p+q}^C(B,B) , \]
for a surjective ring homomorphism $C\to B$.\footnote{The same day we are preparing this paper to upload it to arXiv, it has appeared in that site the paper "A Hochschild-Kostant-Rosenberg theorem and residue sequences for logarithmic Hochschild homology" by Federico Binda, Tommy Lundemo, Doosung Park and Paul Arne Østvær (arXiv:2209.14182) where the authors give an analogous spectral sequence in the case of Hochschild homology, that is, when $C\to B$ is the diagonal homomorphism of a homomorphism. In that paper, they even compute all the terms E$^2_{pq}$, while we treat only the terms with $q< 2$.}

We will use the notation on the logarithmic setting as in \cite[Sections 2 and 3]{CM}, and on the logarithmic cotangent complex as in \cite[Section 4]{CM}. For instance, a \emph{prelog ring} $(A,M,\alpha)$ (or simply $(A,M)$ if there is no confusion) consists of a commutative ring $A$, a commutative monoid $M$ and a multiplicative homomorphism of monoids $\alpha\colon M\to A$. When $m\in M$, we sometimes write also $m$ for its image $\alpha (m)$ in $A$. A homomorphism of prelog rings \[f=(f^\sharp,f^\flat)\colon  (A,M,\alpha_A)\longrightarrow (B,N,\alpha_B)\] is a homomorphism of rings $f^\sharp\colon A\to B$ together with a homomorphism of monoids $f^\flat\colon M\to N$ such that $\alpha_B f^\flat =f^\sharp \alpha_A$.

\section{A complex to compute $H_2$}\label{S::Compute}
	\noindent

		Let $(A,M)\to (B,N)$ be a homomorphism of prelog rings. Let $(A,M)\to (R,P_0)\to (B,N)$ be a factorization with $R\to B$ and $h\colon P_0\to N$ surjective homomorphisms, where $P_0=M\oplus\N^X$ (with $h$ extending $M\to N$) and $R=\Z[P_0]\otimes_{\Z[M]}A[Y]= A[X\cup Y]$ (with $R\to B$ extending $A\to B$).		
		
		We consider $W_0=\ker(P_0^{\gp}\to N^{\gp})$, $J=\ker(\Z[P_0]\to \Z[N])$ and $I=\ker(R\to B)$. Let
		\[ 0\longrightarrow V \longrightarrow G \overset{\pi}{\longrightarrow}J \longrightarrow 0 \]
		be an exact sequence of $\Z[P_0]$-modules with $G$ a free $\Z[P_0]$-module,
		\[ 0\longrightarrow U \longrightarrow F \overset{\tau}{\longrightarrow}I \longrightarrow 0 \]
		and exact sequence of $R$-modules with $F$ a free $R$-module containing $R\otimes_{\Z[P_0]}G$ as a direct summand and $\tau$ extending $G\overset{\pi}{\to}J\to I$, and
		\[ 0 \longrightarrow W_1 \longrightarrow Q_1 \longrightarrow W_0 \longrightarrow 0 \]
		an exact sequence of $\Z$-modules with $Q_1$ a free $\Z$-module.
		
		Let $U_0=\{ \tau(x)y-\tau(y)x\in U \; | \; x,y\in F \}$, which is an  $R$-submodule of $U$, and $V_0=\{ \pi(x)y-\pi(y)x\in V \; | \; x,y\in G \}$ which is a $\Z[P_0]$-submodule of $V$. The $R$-module $U/U_0$ is in fact a $B$-module since $IU\subset U_0$ and similarly, $V/V_0$ is a $\Z[N]$-module.

Consider the following commutative diagram of $B$-modules
		\begin{equation}\label{LSdiagram}
		\begin{tikzpicture}[baseline= (a).base]
		\node[scale=0.9] (a) at (0,0){
			\hspace{-1.8cm}
			\begin{tikzcd}[row sep=1.8em, column sep=-4em]
			\qquad\quad B\otimes_{\Z[N]}V/V_0 \qquad \ar[shorten <= -2em]{rrr}{\beta_2} \ar{ddd}{\alpha_2} \ar{dr}{d_2}&[-1.75em] & &[-0.5cm] B\otimes_{\Z}W_1 \ar{dr}{D_2}&[5em] &[4em] \\
			& B\otimes_{\Z[P_0]}G=B\otimes_{\Z[N]}G/JG \ar{rrr}{\beta_1} \ar{ddd}{\alpha_1} \ar{dr}{d_1}& & & B\otimes_\Z Q_1\ar{dr}{D_1} & \\
			&  & B\otimes_{\Z[P_0]}\Omega_{\Z[P_0]|\Z[M]} \ar{rrr}{\beta_0} \ar{ddd}{\alpha_0} & & & B\otimes_\Z P_0^\gp/M^\gp\\
			U/U_0 \ar{dr}{\delta_2}& & & & & \\
			& F/IF \ar{dr}{\delta_1}& & & & \\
			& & B\otimes_R\Omega_{R|A} & & &
			\end{tikzcd}
			};
			\end{tikzpicture}
		\end{equation}
		where the homomorphisms are defined as follows:
		\begin{itemize}
			\item $\delta_2$, $d_2$ and $D_2$ are induced by the inclusions $U\to F$, $V\to G$ and $W_1\to Q_1$ respectively.
			
			\item $\delta_1$ is the composition
			$F/IF \to I/I^2 \to B\otimes_R\Omega_{R|A}$,
			where the first map is the one induced by $\tau$ and the second one by the canonical derivation $ R\to \Omega_{R|A}$. The map $d_1$ is analogously defined while $D_1$ is induced by the composition
			$Q_1\to W_0\to P_0^\gp\to P_0^\gp/ M_0^\gp$.
			
			\item The homomorphism $\beta_0$ is the one of \cite[Definition 3.9]{CM}: $\beta_0(1\otimes dp):=h(p)\otimes \bar{p}$.
			
			\item The exact sequences
			\[ 0\longrightarrow V\longrightarrow G\longrightarrow J\longrightarrow 0 ,\]
			\[ 0 \longrightarrow W_1 \longrightarrow Q_1 \longrightarrow W_0 \longrightarrow 0 \]
			give a diagram of homomorphisms of $\Z[N]$-modules
			\[
			\begin{tikzcd}[column sep=1.8em,row sep=3.8em]
			& V/JV \arrowr & G/JG \arrowr &
			J/J^2 \arrowr\arrowd{v_h} & 0 \\
			0 \arrowr & \Z[N]\otimes_\Z W_1 \arrowr & \Z[N]\otimes_\Z Q_1 \arrowr & \Z[N]\otimes_\Z W_0 \arrowr & 0
			\end{tikzcd}
			\]
			where $v_h$, which is defined by
			\[ \overline{\lambda(p-p')}\longmapsto \lambda h(p')\otimes p(p')^{-1} \]
			(whenever $\lambda \in \Z$, $h(p)=h(p')$), is the map of \cite[Definition 3.1]{CM} (see \cite[Remark 3.2]{CM}). Since $G/JG= G\otimes_{\Z[P_0]}\Z[N]$ is a projective $\Z[N]$-module, we have a commutative diagram
			\[
			\hspace{-1cm}\begin{tikzcd}[column sep=2em,row sep=4em]
			& V/JV \arrowr{f} \arrowd[dashed]{\beta'_2} & G/JG \arrowr \arrowd{\beta'_1} &
			J/J^2 \arrowr\arrowd{v_h} & 0 \\
			0 \arrowr & \Z[N]\otimes_\Z W_1 \arrowr{g} & \Z[N]\otimes_\Z Q_1 \arrowr & \Z[N]\otimes_\Z W_0 \arrowr & 0
			\end{tikzcd}
			\]
			
			We define $\beta_1=B\otimes_{\Z[N]}\beta'_1$. Since $V_0\subset JG$, $f(V_0/JV)=0$ and then $\beta'_2(V_0/JV)=0$ by the injectivity of $g$. So $\beta'_2$ gives a map $\tilde{\beta'_2}\colon V/V_0\to\Z[N]\otimes_\Z W_1$ and we define $\beta_2=B\otimes_{\Z[N]}\tilde{\beta}_0$.
			
			\item Finally, $\alpha_2,\alpha_1,\alpha_0$ are the obvious maps.
		\end{itemize}
		The commutativity of the square $D_1\beta_1=\beta_0d_1$ follows from the commutativity of the squares
		\[
		\begin{tikzcd}[column sep=3.5em,row sep=4em]
		G/JG \arrowr\arrowd{\beta_1'} & J/J^2 \arrowd{v_h} \\
		\Z[N]\otimes_\Z Q_1  \arrowr & \Z[N]\otimes_\Z W_0
		\end{tikzcd}
		\]
		and
		\[
		\begin{tikzcd}[column sep=3.5em,row sep=4em]
		J/J^2 \arrowr\arrowd{v_h} & \Z[N]\otimes_{\Z[P_0]}\Omega_{\Z[P_0]|\Z[M]} \arrowd \\
		\Z[N]\otimes_\Z W_0  \arrowr & \Z[N]\otimes_\Z P_0^\gp/M^\gp
		\end{tikzcd}
		\]
where the right map in this square is defined similarly to $\beta_0$.

		The commutativity of the remaining squares is clear.

We have $\delta_1\delta_2 =0$, $d_1d_2 =0$, $D_1D_2 =0$.

		\begin{defn}
			We define the complex $\LS_{(B,N)|(A,M)}$ as follows: $(\LS_{(B,N)|(A,M)})_i$ is the pushout of $(\alpha_i, \beta_i)$ for $i\in \{0,1,2\}$ and zero for $i\notin \{0,1,2\}$, and the differential is the one induced by the maps $\delta_i$, $d_i$, $D_i$.
		\end{defn}
		
		\begin{prop}\label{1.2}
			In this context, we consider the complex
			\[K:= \big(T\otimes_\Z W_1 \overset{D_2}{\longrightarrow} T\otimes_\Z Q_1 \overset{D_1}{\longrightarrow} T\otimes_\Z P_0^\gp/M^\gp\big)\]
			for any $\Z$-module $T$. Then:
			\begin{enumerate}
				\item[(i)] $H_2(K)=\Tor_1^\Z\big(T,\ker(M^\gp\to N^\gp)\big)$
				\item[(ii)] We have an exact sequence
				\[ \hspace{-1cm} 0\to T\otimes_\Z \ker(M^\gp\to N^\gp)\to H_1(K)\to \Tor_1^\Z(T,N^\gp/\im(M^\gp\to N^\gp)) \to 0. \]
				\item[(iii)] $H_0(K)=T\otimes_\Z N^\gp/\im(M^\gp\to N^\gp)$
			\end{enumerate}
		\end{prop}
		\begin{proof}
			We have $H_2(K)=\ker(T\otimes_\Z W_1 \to T\otimes_\Z Q_1)$. Applying $T\otimes_\Z -\ $ to the exact sequence of $\Z$-modules
			\[ 0 \longrightarrow W_1 \longrightarrow Q_1 \longrightarrow W_0 \longrightarrow 0 \]
			we obtain an exact sequence
			\[0=\Tor_1^\Z(T,Q_1)\to \Tor_1^\Z(T,W_0) \to T\otimes_\Z W_1 \to T\otimes_\Z Q_1 \to T\otimes_\Z W_0 \to 0 \]
			and then $H_2(K)=\Tor_1^\Z(T,W_0)$.
			
			Consider the diagram of exact rows and columns defining the lower row
			\begin{equation}\label{diagram_4.3.(2)}
				\hspace{-0.25em}\begin{tikzcd}[column sep=0.2em,row sep=3em]
				&[1em] 0\arrowd & 0\arrowd & 0\arrowd &[1em] \\
				0 \arrowr & \ker(M^\gp\to N^\gp) \arrowr\arrowd & M^\gp \arrowr\arrowd &
				\im(M^\gp\to N^\gp) \arrowr\arrowd & 0 \\
				0 \arrowr & W_0 \arrowr\arrowd & P_0^\gp=M^\gp\oplus\Z^X \arrowr\arrowd & N^\gp \arrowr\arrowd & 0 \\
				0 \arrowr & W_0/\ker(M^\gp\to N^\gp) \arrowr\arrowd & \Z^X \arrowr\arrowd &
				N^\gp/\im(M^\gp\to N^\gp) \arrowr\arrowd & 0 \\
				& 0 & 0 & 0 &
				\end{tikzcd}
			\end{equation}
			Since the lower row is exact,
			\[ \Tor_i^\Z\big(T,W_0/\ker(M^\gp\to N^\gp)\big) = \Tor_{i+1}^\Z\big(T,N^\gp/\im(M^\gp\to N^\gp)\big)\]
			and this last module vanishes for $i+1\geq 2$ given that $\Z$ is a principal ideal\linebreak domain. So, from the left column we deduce
			\[ H_2(K)= \Tor_{1}^\Z(T,W_0)= \Tor_1^\Z\big(T,\ker(M^\gp\to N^\gp)\big) . \]
			
			Consider now the diagram
			\[
			\hspace{-0.5em}\begin{tikzcd}[column sep=0.78em,row sep=3.5em]
			& & & 0\arrowd & & &\\
			& & & \ker(M^\gp\to N^\gp) \arrowd & & &\\
			0 \arrowr & W_1 \arrowr & Q_1 \arrowr{\varphi}\arrow[dashed]{dr}{\sigma} & W_0 \arrowr\arrowd & 0 & &\\
			& & 0 \arrowr & W_0/\ker(M^\gp\to N^\gp) \arrowr\arrowd & \Z^X \arrowr & N^\gp/\im(M^\gp\to N^\gp) \arrowr& 0\\
			& & & 0 & & &
			\end{tikzcd}
			\]
			We have exact sequences
			\begin{equation}\label{diagram_4.3.(3)}
			0 \hspace{-0.06em}\to\hspace{-0.06em} \varphi^{-1}(\ker(M^\gp\hspace{-0.2em}\to\hspace{-0.2em} N^\gp)) \hspace{-0.06em}\to\hspace{-0.06em} Q_1\hspace{-0.06em}\to\hspace{-0.06em} \Z^X \to N^\gp/\im(M^\gp\hspace{-0.2em}\to\hspace{-0.2em} N^\gp) \hspace{-0.06em}\to\hspace{-0.06em} 0 ,
			\end{equation}
			and
			\[ 0 \to W_1 \to \varphi^{-1}(\ker(M^\gp\to N^\gp)) \to \ker(M^\gp\to N^\gp) \to 0 . \]
			Therefore we have a vertical exact sequence of (horizontal) complexes
			\[
			\begin{tikzcd}[column sep=2.5em,row sep=3em]
			& & &[-0.5em] \;\arrow[bend right=20,swap,dash]{dddddd} &[-3.2em] &[-4.5em] \;\arrow[bend left=20,swap,dash]{dddddd} \\[-3em]
			0 \arrowd & 0\arrowd & 0\arrowd & & 0\arrowd & \\
			W_1 \arrowd\arrowr & Q_1 \arrowr\arrow[equal]{d} & \Z^X \arrow[shorten <= 1.5em]{rr}\arrow[equal]{d} & & N^\gp/\im(M^\gp\to N^\gp) \arrow[equal]{d} & \\
			\varphi^{-1}(\ker(M^\gp\to N^\gp)) \arrowd\arrowr & Q_1 \arrowr\arrowd & \Z^X \arrow[shorten <= 1.5em]{rr}\arrowd & & N^\gp/\im(M^\gp\to N^\gp) \arrowd & \\
			\ker(M^\gp\to N^\gp) \arrowd \arrowr & 0 \arrowr\arrowd & 0 \arrow[shorten <= 2.5em]{rr}\arrowd & &
			0 \arrowd & \\
			0 & 0 & 0 & & 0 & \\[-3em]
			& & & \; & & \;
			\end{tikzcd}
			\]
			Applying $T\otimes_\Z-$ we obtain a diagram of (horizontal) complexes
			\begin{equation}\label{diagram_4.3.(4)}
				\begin{tikzcd}[column sep=2em,row sep=3em]
					E'\colon \quad &[-3em] T\otimes_\Z W_1 \arrowr\arrowd & T\otimes_\Z Q_1 \arrowr\arrow[equal]{d} & T\otimes_\Z \Z^X \arrow[equal]{d}  \\
					E\colon \quad & T\otimes_\Z \varphi^{-1}(\ker (M^\gp \to N^\gp )) \arrowr\arrowd & T\otimes_\Z Q_1 \arrowr & T\otimes_\Z \Z^X  \\
					E''\colon \quad & T\otimes_\Z \ker (M^\gp \to N^\gp ) \arrowd & & \\
					& 0 & &
				\end{tikzcd}
			\end{equation}
			where the columns correspond to degrees 2, 1 and 0 respectively (from left to right).
			
			We have
			\[ H_n(E)= \Tor_n^\Z\big(T,N^\gp / \im(M^\gp \to N^\gp ) \big) \]
			for $n=0,1$ by the exactness of \eqref{diagram_4.3.(3)}, and the fact that $Q_1$ and $\Z^X$ are free $\Z$-modules.
			
			On the other hand, $W_0/\ker (M^\gp \to N^\gp )$ is a free $\Z$-module	by the lower row of \eqref{diagram_4.3.(2)}, and then the exact sequence
			\[ 0 \longrightarrow \varphi^{-1}(\ker (M^\gp \to N^\gp )) \longrightarrow Q_1 \overset{\sigma}{\longrightarrow} W_0/\im(M^\gp \to N^\gp ) \longrightarrow 0 \]
			splits. Therefore
			\[ 0 \longrightarrow T\otimes_\Z\varphi^{-1}(\ker (M^\gp \to N^\gp )) \longrightarrow T\otimes_\Z Q_1 \]
			is injective and then $H_2(E)=0$.
			
			Consider the left column of \eqref{diagram_4.3.(4)} and let
			\begin{align*}
				\Lambda &:= \ker\big[ T\otimes_\Z\varphi^{-1}(\ker (M^\gp \to N^\gp )) \to T\otimes_\Z\ker (M^\gp \to N^\gp)\big]\\
				&\; = \im \big[ T\otimes_\Z W_1 \to T\otimes_\Z\varphi^{-1}(\ker (M^\gp \to N^\gp))\big] \\
				&\;  = \im (T\otimes_\Z W_1 \to T\otimes_\Z Q_1)     .
			\end{align*}
			
			We have a vertical exact sequence of (horizontal) complexes
		\[
		\begin{tikzcd}[column sep=2em,row sep=3.1em]
		&[-3em] 0\arrowd & 0\arrowd & 0\arrowd \\
		\tilde{E'}\colon \quad & \Lambda \arrowr\arrowd & T\otimes_\Z Q_1 \arrowr\arrow[equal]{d} & T\otimes_\Z \Z^X \arrow[equal]{d}  \\
		E\colon \quad & T\otimes_\Z \varphi^{-1}(\ker (M^\gp \to N^\gp )) \arrowr\arrowd & T\otimes_\Z Q_1 \arrowr\arrowd & T\otimes_\Z \Z^X\arrowd \\
		E''\colon \quad & T\otimes_\Z \ker (M^\gp \to N^\gp ) \arrowr\arrowd & 0 \arrowr & 0 \\
		& 0 & &
		\end{tikzcd}
		\]
		and taking the homology exact sequence
			\[ 0 = H_2(E) \to H_2(E'')= T\otimes_\Z\ker (M^\gp \to N^\gp ) \to H_1(\tilde{E'})= H_1(E') \to \]
			\[ \to H_1(E)=\Tor_1^\Z(T,N^\gp/\im(M^\gp\to N^\gp)) \to H_1(E'')=0 \]
			we obtain the desired exact sequence
			\[ 0\to T\otimes_\Z \ker(M^\gp\to N^\gp)\to H_1(K)\to \Tor_1^\Z(T,N^\gp/\im(M^\gp\to N^\gp)) \to 0. \]
			
			Finally, $H_0(K)= H_0(E)= T\otimes_\Z N^\gp/\im(M^\gp\to N^\gp)$.			
		\end{proof}

\begin{thm}
The homology modules of the logarithmic cotangent complex of Gabber can be computed as the homology of the complex $\LS_{(B,N)|(A,M)}$.
\end{thm}
\begin{proof}
We keep the notation as in the beginning of this section. Let $\LS_{\Z[N]|\Z[M]}$ be the complex of $\Z[N]$-modules
\[ V/V_0 \longrightarrow G/JG \longrightarrow \Z[N]\otimes_{\Z[P_0]}\Omega_{\Z[P_0]|\Z[M]} , \]
$\LS_{B|A}$ the complex of $B$-modules
\[ U/U_0 \longrightarrow F/IF \longrightarrow B\otimes_R\Omega_{R|A} \]
and $D$ the complex of $\Z$-modules
\[ W_1 \longrightarrow Q_1 \longrightarrow P_0^\gp/M^\gp . \]
By \cite{LS}, for any $B$-module $T$, $H_*(\LS_{\Z[N]|\Z[M]}\otimes_{\Z[N]}T)$ and $H_*(\LS_{B|A}\otimes_BT)$ do not depend on the choices, and similarly $H_*(D\otimes_\Z T)$ by Proposition~\ref{1.2}.

By definition, we have a commutative diagram of complexes
\[
\begin{tikzcd}[column sep=3em,row sep=4em]
B\otimes_{\Z[N]}\LS_{\Z[N]|\Z[M]} \arrowr{\beta} \arrowd{\alpha} & B\otimes_\Z D \arrowd{\varepsilon} \\
\LS_{B|A}  \arrowr{\gamma} & \LS_{(B,N)|(A,M)}
\end{tikzcd}
\]
which is a pushout in each degree, and so $\coker(\alpha_i)=\coker(\varepsilon_i)$.

Therefore we have a commutative diagram of complexes for any $B$-module $T$
\[
\begin{tikzcd}[column sep=1.2em,row sep=4em]
0 \arrowr & T\otimes_{\Z[N]}\LS_{\Z[N]|\Z[M]} \arrowr{T\otimes\alpha}\arrowd{T\otimes\beta} & T\otimes_B\LS_{B|A} \arrowr\arrowd{T\otimes\gamma} &
\coker(T\otimes\alpha) \arrowr\arrow[equal]{d} & 0 \\
0 \arrowr & T\otimes_\Z D\arrowr{T\otimes\varepsilon} & T\otimes_B\LS_{(B,N)|(A,M)} \arrowr & \coker(T\otimes\varepsilon) \arrowr & 0
\end{tikzcd}
\]
Since $\alpha_1, \alpha_0$ are split injective, so are $\varepsilon_1, \varepsilon_0$, and then both lines are exact except that
\[T\otimes_{\Z[N]}(\LS_{\Z[N]|\Z[M]})_2= T\otimes_{\Z[N]}V/V_0 \longrightarrow T\otimes_{B}U/U_0= T\otimes_{B}(\LS_{B|A})_2 \]
and
\[ T\otimes_{\Z}D_2 = T\otimes_\Z W_1 \longrightarrow T\otimes_{B}(\LS_{(B,N)|(A,M)} )_2 \]
are not necessarily injective. We deduce a commutative diagram of exact rows
\[
\hspace{-0.2cm}\begin{tikzcd}[column sep=1em,row sep=4em]
H_2(T\otimes_{\Z[N]}\LS_{\Z[N]|\Z[M]}) \arrowr\arrowd & H_2(T\otimes_B\LS_{B|A}) \arrowr\arrowd &
H_2(\coker(T\otimes\alpha)) \arrowr \arrow[equal]{d} & \quad \\
H_2(T\otimes_{\Z}D) \arrowr & H_2(T\otimes_B\LS_{(B,N)|(A,M)}) \arrowr &
H_2(\coker(T\otimes\varepsilon)) \arrowr & \quad
\end{tikzcd}
\]
\[
\hspace{1cm}\begin{tikzcd}[column sep=1em,row sep=4em]
\arrowr & H_1(T\otimes_{\Z[N]}\LS_{\Z[N]|\Z[M]}) \arrowr\arrowd & \cdots \arrowr &
H_0(\coker(T\otimes\alpha)) \arrowr\arrow[equal]{d} & 0 \\
\arrowr & H_1(T\otimes_{\Z}D) \arrowr & \cdots \arrowr &
H_0(\coker(T\otimes\varepsilon)) \arrowr & 0
\end{tikzcd}
\]
By diagram chasing, we deduce an exact sequence
\[\begin{tikzpicture}[baseline= (a).base]
\node[scale=0.85] (a) at (0,0){
	\hspace{-2em}\begin{tikzcd}[column sep=0.5em,row sep=0.5ex]
	\, & H_2(T\otimes_{\Z[N]}\LS_{\Z[N]|\Z[M]}) \arrowr & H_2(T\otimes_B\LS_{B|A})\oplus H_2(T\otimes_\Z D) \arrowr & H_2(T\otimes_B\LS_{(B,N)|(A,M)})  \arrowr & \phantom{0}
	\end{tikzcd}
};
\end{tikzpicture}\]
\vspace{-1em}
\[\begin{tikzpicture}[baseline= (a).base]
\node[scale=0.85] (a) at (0,0){
	\hspace{-1.3em}\begin{tikzcd}[column sep=0.5em,row sep=0.5ex]
	\, \arrowr & H_1(T\otimes_{\Z[N]}\LS_{\Z[N]|\Z[M]}) \arrowr & \qquad\qquad\qquad \cdots \qquad\qquad\qquad\arrowr & H_0(T\otimes_B\LS_{(B,N)|(A,M)})  \arrowr & 0.
	\end{tikzcd}
};
\end{tikzpicture}\]

By \cite[Proposition~15.12]{An-1974} and Proposition~\ref{1.2}, this exact sequence takes the form
\[\begin{tikzpicture}[baseline= (a).base]
\node[scale=0.9] (a) at (0,0){
	\hspace{-1.2em}\begin{tikzcd}[column sep=0.8em,row sep=0.5ex]
	\, & H_2(\Z[M],\Z[N],T) \arrowr & H_2(A,B,T)\oplus H_2(T\otimes_\Z D) \arrowr & H_2(T\otimes_B\LS_{(B,N)|(A,M)})  \arrowr & \phantom{0}\\
	\, \arrowr & H_1(\Z[M],\Z[N],T) \arrowr & H_1(A,B,T)\oplus H_1(T\otimes_\Z D) \arrowr & H_1(T\otimes_B\LS_{(B,N)|(A,M)})  \arrowr & \phantom{0}\\
	\, \arrowr & H_0(\Z[M],\Z[N],T) \arrowr & H_0(A,B,T)\oplus H_0(T\otimes_\Z D) \arrowr & H_0(T\otimes_B\LS_{(B,N)|(A,M)})  \arrowr & 0
	\end{tikzcd}
};
\end{tikzpicture}\]
where
\begin{itemize}
	\item $H_2\big(D\otimes_\Z W)= \Tor_1^\Z(T,\ker(M^\gp\to N^\gp)\big)$,
	\item $H_1(D\otimes_\Z W)$ appears in an exact sequence
	\[\hspace{-1.05cm}
	\begin{tikzcd}[row sep=0em, column sep=0.3em]
	0 \arrow[shorten <= -0.2em,shorten >= -0.2em]{r} & T\otimes_\Z\ker(M^\gp\hspace{-0.1em}\to \hspace{-0.1em}N^\gp)\arrow[shorten <= -0.2em,shorten >= -0.2em]{r} & H_1(D\otimes_\Z W) \arrow[shorten <= -0.2em,shorten >= -0.2em]{r} & \Tor_1^\Z\big(T,\coker(M^\gp\hspace{-0.1em}\to \hspace{-0.1em}N^\gp)\big) \arrow[shorten <= -0.2em,shorten >= -0.2em]{r} & 0 ,
	\end{tikzcd}\]
	\item $H_0(D\otimes_\Z W)= T\otimes_\Z\coker(M^\gp\to N^\gp)$.
\end{itemize}

Let $(A,M)\overset{\varphi}{\longrightarrow} (F,R) \overset{\psi}{\longrightarrow} (B,N)$ be a factorization in the category of\linebreak simplicial prelog rings with $\varphi$ a free cofibration and $\psi$ a trivial fibration as in \cite[Section 4]{CM}. We have a commutative diagram of complexes
\[\begin{tikzpicture}[baseline= (a).base]
\node[scale=1] (a) at (0,0){
	\begin{tikzcd}[row sep=2.5em, column sep=-1em]
	T\otimes_{\Z[R]}\Omega_{\Z[R]|\Z[M]} \ar{dd}  \ar{dr} \ar{rr} & & T\otimes_\Z R^\gp/M^\gp \ar{dd}\ar{dr} & \\
	& T\otimes_{\Z[N]}\LS_{\Z[N]|\Z[M]} \ar[crossing over]{rr} & & T\otimes_{\Z}D \ar{dd}\\
	T\otimes_{F}\Omega_{F|A} \ar{rr} \ar{dr} & & T\otimes_B\LLL_{(B,N)|(A,M)} \ar{dr} & \\
	& T\otimes_B\LS_{B|A} \ar{rr} \ar[crossing over, leftarrow]{uu}  & & T\otimes_B\LS_{(B,N)|(A,M)}
	\end{tikzcd}
};
\end{tikzpicture}\]
inducing a morphism from the exact sequence of \cite[Theorem 4.3]{CM} into the above\linebreak exact sequence, and giving isomorphisms in two of each three terms (by \cite[Proposition~15.12]{An-1974} the ones corresponding to the oblique maps, and by direct inspection the one associated to the oblique upper right map, since we can take as $P_0^{gp}$, $Q$, the groups $R_0^{gp}$, $R_1^{gp}$ of the proof of \cite[Theorem 4.3]{CM}). Therefore, the remaining homomorphisms
\[ H_i(T\otimes_B\LLL_{(B,N)|(A,M)}) \longrightarrow H_i(T\otimes_B\LS_{(B,N)|(A,M)}) , \qquad \text{for } i=0,1,2, \]
are also isomorphisms.
\end{proof}

\section{A spectral sequence}\label{S::SS}
	\noindent
	
	Let $C\to B$ be a surjective homomorphism of rings. We have a convergent spectral sequence \cite[Theorem 6.8]{Quillen-MIT}
	\[ E_{p,q}^2 = H_{p+q}(S^q\LLL_{B|C}) \Rightarrow \Tor_{p+q}^C(B,B) , \]
	which when $A\to B$ is a flat homomorphism and $C=B\otimes_A B$ takes the form
	\[ E_{p,q}^2 = H_{p}(\wedge^q\LLL_{B|A}) \Rightarrow \HHH_{p+q}(B|A) \]
	where $\HHH_*(B|A)=\Tor_{p+q}^{B\otimes_A B}(B,B)$ is the Hochschild homology of the $A$-algebra $B$ and $S^q$, $\wedge^q$ denote symmetric and exterior powers respectively. We will give here an analogous spectral sequence for the logarithmic cotangent complex.
	
	Let $(C,Q)\to (B,N)$ be a homomorphism of prelog rings such that $C\to B$ and $Q\to N$ are surjective. Let $(C,Q)\to (R,P)\to (B,N)$ be a free cofibration\,-\,trivial fibration factorization as in \cite[Section 4]{CM} with $R_0=C$ and $P_0=Q$. Let $I=\ker(\pi: R\otimes_CB\to B)$, $J=\ker(\Z[P]\otimes_{\Z[Q]}\Z[N]\to\Z[N])$ and $T=\ker((P\oplus_Q N)^\gp \to N^\gp )$.
	
	Let $D_{(B,N)|(C,Q)}^1$ be the simplicial $R\otimes_CB$-module defined by the pushout
	\[
	\begin{tikzcd}[column sep=3em,row sep=4em]
	U \arrowr{\beta}\arrowd{\alpha} & B\otimes_\Z T \arrowd \\
	I \arrowr & D_{(B,N)|(C,Q)}^1
	\end{tikzcd}
	\]
	where $U=(R\otimes_CB)\otimes_{\Z[P]\otimes_{\Z[Q]}\Z[N]} J$, $\beta$ is the homomorphism induced by the map $J\to B\otimes_\Z T$ of the proof of \cite[Proposition 3.1]{CM} and the map $\pi\colon R\otimes_CB\to B$, and the homomorphism $\alpha$ is the obvious one.
	
	\begin{prop}\label{4.18}
		In this situation, we have an exact sequence
		\[\cdots \longrightarrow H_n(U) \longrightarrow  \Tor_{n}^C(B,B)\oplus W_n  \longrightarrow H_n\big(D_{(B,N)|(C,Q)}^1\big) \longrightarrow \quad\]
		\[  H_{n-1}(U) \longrightarrow  \qquad\cdots \qquad \longrightarrow H_1\big(D_{(B,N)|(C,Q)}^1\big) \longrightarrow 0\]
		where
		\[
		W_n =
		\begin{cases}
		\ker (Q^\gp \to N^\gp)\otimes_Z B  &\quad\text{if } n=1, \\
		\Tor_1^\Z (\ker (Q^\gp \to N^\gp),B)  &\quad\text{if } n=2, \\
		0  &\quad\text{if } 1\neq n \neq 2.
		\end{cases}
		\]
	\end{prop}
	\begin{proof}
		If $P=Q\oplus \N^Y$, $R=C[X\cup Y]$, then $J$ is the ideal of $\Z[P]\otimes_{\Z[Q]}\Z[N]=\Z[N][Y]$ generated by $Y$, $I$ is the ideal of $R\otimes_CB= B[X\cup Y]$ generated by $X\cup Y$, and the map $\alpha$ in the diagram is (split) injective. Thus, we have an exact sequence
		\[ \cdots \to H_n(U) \to H_n(I)\oplus H_n(B\otimes_\Z T) \to H_n\big(D_{(B,N)|(C,Q)}^1\big) \to H_{n-1}(U) \to \cdots \ .\]
		
		Since $0\to I\to R\otimes_CB\to B\to 0$ is exact and $H_n(R\otimes_CB)= \Tor_n^C(B,B)$, we have $H_n(I)= \Tor_n^C(B,B)$ if $n>0$ and $H_0(I)=0$ (in fact, $I_0=0$). We also have $J_0=0$ and then $H_0(U)=0$.
		
		It remains to compute $H_n(B\otimes_\Z T)$. Since
		\[P^\gp \oplus_{Q^\gp}N^\gp = P^\gp/\ker(Q^\gp \to N^\gp),\] we have an exact sequence
		\begin{equation}\label{4.2.1}
			0 \longrightarrow T \longrightarrow P^\gp / \ker(Q^\gp \to N^\gp) \longrightarrow N^\gp \longrightarrow 0
		\end{equation}
		
		From the exact sequence
		\[ 0 \longrightarrow \ker(Q^\gp \to N^\gp) \longrightarrow P^\gp \longrightarrow P^\gp / \ker(Q^\gp \to N^\gp) \longrightarrow 0 \]
		we deduce $H_n(P^\gp / \ker(Q^\gp \to N^\gp) )=0$ for $n\geq 2$ and an exact sequence
		\[\hspace{4cm}0 \longrightarrow H_1\big(P^\gp / \ker(Q^\gp \to N^\gp)\big) \longrightarrow \]
		\[\hspace{-1cm} \ker(Q^\gp \to N^\gp) \longrightarrow N^\gp \longrightarrow H_0\big(P^\gp / \ker(Q^\gp \to N^\gp)\big)\longrightarrow 0 \]
		showing that
		\[ H_1\big(P^\gp / \ker(Q^\gp \to N^\gp) \big) = \ker(Q^\gp \to N^\gp) ,\]
		\[ H_0\big(P^\gp / \ker(Q^\gp \to N^\gp)\big) = N^\gp .\]
		
		Then, from the exact sequence \eqref{4.2.1} we obtain $H_n(T)=0$ for all $n\geq 2$, $H_1(T)=\ker(Q^\gp \to N^\gp)$ and an exact sequence
		\[ 0 \longrightarrow H_0(T) \longrightarrow N^\gp \xrightarrow{=} N^\gp \longrightarrow 0\]
		showing that $H_0(T)=0$.
		
		By the universal coefficient theorem, we deduce
		\[
		H_n(B\otimes_\Z T)=
		\begin{cases}
		H_1(T)\otimes_\Z B=\ker (Q^\gp \to N^\gp)\otimes_Z B  &\quad\text{if } n=1, \\
		\Tor_1^\Z (\ker (Q^\gp \to N^\gp),B)  &\quad\text{if } n=2, \\
		0  &\quad\text{if } 1\neq n \neq 2.
		\end{cases}
		\]
	\end{proof}
	
	\begin{ex}
		We are going to see that $H_1(D_{(B,N)|(C,Q)}^1)$ is the conormal module $N_{(B,N)|(C,Q)}$ defined in \cite[Section 3]{CM}.
		
		Let $\Z[P]\otimes_{\Z[Q]}\Z[N]\to X\to R\otimes_CB$ be a cofibration\,-\,trivial fibration factorization, and consider the derived tensor product
		\[\tilde{U}=X\otimes_{\Z[P]\otimes_{\Z[Q]}\Z[N]}J=(R\otimes_CB)\overset{L}{\otimes}_{\Z[P]\otimes_{\Z[Q]}\Z[N]}J .\] We have an exact sequence defining $Z$
		\[0 \longrightarrow Z \longrightarrow \tilde{U} \longrightarrow U \longrightarrow 0 \]
		and $\tilde{U}_0=0=U_0$ since $J_0=0$. Therefore $Z_0=0$ and then $H_1(\tilde{U})\to H_1(U)$ is surjective. So from Proposition~\ref{4.18} we have an exact sequence
		\[ H_1(\tilde{U}) \longrightarrow \aaa/\aaa^2\oplus (\ker(Q^\gp \to N^\gp)\otimes_\Z B) \longrightarrow H_1\big(D_{(B,N)|(C,Q)}^1\big) \longrightarrow 0 \]
		where $\aaa=\ker(C\to B)$.
		
		Let us compute $H_1(\tilde{U})$. By \cite[II, Theorem 6]{Quillen-HA} we have a spectral sequence
		\[ E_{p,q}^2=\Tor_{p}^{\Tor_{*}^{\Z[Q]}(\Z[N],\Z[N])}\big(\Tor_*^C(B,B),H_*(J)\big)_q \;\Rightarrow\; H_{p+q}(\tilde{U}). \]
		Moreover, from the $\Tor$ long exact sequence associated to the exact sequence
		\[ 0 \longrightarrow H_*(J) \longrightarrow \Tor_*^{\Z[Q]}(\Z[N],\Z[N]) \longrightarrow \Z[N] \longrightarrow 0 , \]
		we obtain
		\[ E_{p,q}^2=\Tor_{p+1}^{\Tor_{*}^{\Z[Q]}(\Z[N],\Z[N])}\big(\Tor_*^C(B,B),\Z[N]\big)_q \]
		for $p>0$ and an exact sequence
		\[ 0 \longrightarrow \Tor_{1}^{\Tor_{*}^{\Z[Q]}(\Z[N],\Z[N])}\big(\Tor_*^C(B,B),\Z[N]\big)_q \longrightarrow E_{0,q}^2 \longrightarrow \]
		\[ \longrightarrow \Tor_{q}^{C}(B,B) \longrightarrow \big(\Tor_*^C(B,B)\otimes_{\Tor_{*}^{\Z[Q]}(\Z[N],\Z[N])} \Z[N]\big)_q \longrightarrow 0. \]
				
		When $p>0$, we have then
		\begin{align*}
			E_{p,0}^2&=\Tor_{p+1}^{\Tor_{0}^{\Z[Q]}(\Z[N],\Z[N])}\big(\Tor_0^C(B,B),\Z[N]\big)\\
			&=\Tor_{p+1}^{\Z[N]}(B,\Z[N])\\
			&=0
		\end{align*}
		In particular, $E_{1,0}^2=0$ and $E_{2,0}^2=0$. We also have $E_{-2,2}^2=0$, since the spectral sequence is located in the first quadrant. Therefore $E_{0,1}^3=E_{0,1}^2$, and since $E_{0,1}^\infty=E_{0,1}^3$ (again because it is a first quadrant spectral sequence), we deduce $E_{0,1}^\infty=E_{0,1}^2$. Since $E_{1,0}^2=0$, we obtain $E_{1,0}^\infty=0$ and then, since $\HH_0(J)=0$,
		\begin{align*}
		H_1(\tilde{U})&=E_{0,1}^\infty= E_{0,1}^2 = \Tor_{0}^{\Tor_{*}^{\Z[Q]}(\Z[N],\Z[N])}\big(\Tor_*^C(B,B),H_*(J)\big)_1 \\
		& = \big(\Tor_*^C(B,B)\otimes_{\Tor_{*}^{\Z[Q]}(\Z[N],\Z[N])}H_*(J)\big)_1 \\
		& = \Tor_0^C(B,B) \otimes_{\Tor_{0}^{\Z[Q]}(\Z[N],\Z[N])}H_1(J) \\
		& = B\otimes_{\Z[N]}\bbb/\bbb^2
		\end{align*}
		where $\bbb=\ker(\Z[Q]\to \Z[N])$.	
		
		So we have an exact sequence
		\[ B\otimes_{\Z[N]}\bbb/\bbb^2 \longrightarrow \aaa/\aaa^2 \oplus (\ker(Q^\gp \to N^\gp)\otimes_\Z B) \longrightarrow H_1\big(D_{(B,N)|(C,Q)}^1\big) \longrightarrow 0 \]
		where the first map is the sum of the canonical map $B\otimes_{\Z[N]}\bbb/\bbb^2 \to \aaa/\aaa^2$ and the map $B\otimes_{\Z[N]}\bbb/\bbb^2\to \ker (Q^\gp \to N^\gp)\otimes_\Z B$ of \cite[Proposition 3.1]{CM}. By \cite[Definition 3.4]{CM}, we have then
		\[ H_1\big(D_{(B,N)|(C,Q)}^1\big) = N_{(B,N)|(C,Q)} .\]
		
		Note that in the case $C=B\otimes_AB$, $Q=N\oplus_MN$ where $(A,M)\to (B,N)$ is a homomorphism of prelog rings we have
		\[ H_1\big(D_{(B,N)|(C,Q)}^1\big) = \Omega_{(B,N)|(A,M)} .\]
	\end{ex}
	
	\begin{ex}
		If $Q=N$, then $J=T=0$ and so $D_{(B,N)|(C,Q)}^1$ is the pushout
		\[
		\begin{tikzcd}[column sep=3em,row sep=4em]
		0 \arrowr\arrowd & 0 \arrowd \\
		I \arrowr & D_{(B,N)|(C,Q)}^1
		\end{tikzcd}
		\]
		showing that $H_n(D_{(B,N)|(C,Q)}^1)=\Tor_n^C(B,B)$ for $n>0$. In the case where $(A,M)\to (B,N)$ is a homomorphism of prelog rings, with $A\to B$ flat and $M=N$, taking $(C,Q)=(B\otimes_AB,N\oplus_MN)$, we have $H_n\big(D_{(B,N)|(C,Q)}^1\big)=\HHH_n(B|A)$ for $n>0$.
	\end{ex}
	
	\begin{thm}\label{4.21}
		We have a first quadrant convergent spectral sequence
		\[ E_{p,q}^2 \;\Rightarrow\; H_{p+q}\big(D_{(B,N)|(C,Q)}^1\big) \]
		where $E_{p,q}^2=0$ for $q=0$ and $E_{p,1}^2=H_{p+1}\big(\LLL_{(B,N)|(C,Q)}\big)$.
	\end{thm}
	\begin{proof}
		Consider the pushout
		\[
		\begin{tikzcd}[column sep=2em,row sep=4em]
		(R\otimes_CB)\otimes_{\Z[P]\otimes_{\Z[Q]}\Z[N]} J \arrowr{\beta}\arrowd{\alpha} & B\otimes_\Z T \arrowd \\
		I \arrowr & D_{(B,N)|(C,Q)}^1
		\end{tikzcd}
		\]
		We have already seen that $\alpha$ is split injective, and similarly, the map
		\[ (R\otimes_CB)\otimes_{\Z[P]\otimes_{\Z[Q]}\Z[N]} J^i \longrightarrow I^i\]
		is split injective for all $i>0$.
		
		For any $i \geq 2$, let $D^i$ be the image of the composition $I^i\hookrightarrow I \to D_{(B,N)|(C,Q)}^1$.
		
		Applying the Ker-Coker Lemma to the following diagram for $i\geq 2$,
		\begin{equation*}
			\hspace{-2cm}\begin{tikzcd}[column sep=2em,row sep=3em]
			0 \arrowr & (R\otimes_CB)\otimes_{\Z[P]\otimes_{\Z[Q]}\Z[N]}J^i \arrowr\arrowd & I^i \arrowr\arrow[tail]{d}{\epsilon} & \ \\
			0 \arrowr & (R\otimes_CB)\otimes_{\Z[P]\otimes_{\Z[Q]}\Z[N]}J \arrowr & I\oplus(B\otimes_\Z T) \arrowr{\beta} & \
			\end{tikzcd}
		\end{equation*}
		\begin{equation}\label{derradeirastar}
			\hspace{-0.75cm}\begin{tikzcd}[column sep=2em,row sep=3em]
			\ \arrowr & I^i/(R\otimes_CB)\otimes_{\Z[P]\otimes_{\Z[Q]}\Z[N]}J^i \arrow[d,dashed]\arrowr & 0 \\
			\ \arrowr{\beta} & D_{(B,N)|(C,Q)}^1 \arrowr & 0
			\end{tikzcd}
		\end{equation}
		(where the vertical morphism to the left is the obvious one, the one in the middle is the inclusion $I^i\longhookrightarrow I$ and the zero morphism on $B\otimes_\Z T$, and the square to the left is commutative since $i\geq 2$ and the homomorphism $J^2\to B\otimes_\Z T$ is zero) we obtain an exact sequence
		\[ (R\otimes_CB)\otimes_{\Z[P]\otimes_{\Z[Q]}\Z[N]} J /J^i \to I/I^i\oplus(B\otimes_\Z T) \to  D_{(B,N)|(C,Q)}^1/D^i \to 0 , \]
		which we can write as a pushout:
		\begin{equation}\label{sstar}
			\begin{tikzcd}[column sep=3em,row sep=4em]
			(R\otimes_CB)\otimes_{\Z[P]\otimes_{\Z[Q]}\Z[N]} J /J^i \arrowr\arrowd & B\otimes_\Z T \arrowd \\
			I/I^i \arrowr & D_{(B,N)|(C,Q)}^1/D^i
			\end{tikzcd}
		\end{equation}
		
		From the commutative diagram \eqref{derradeirastar} we deduce that $\ker \big(I^i\to D_{(B,N)|(C,Q)}^1\big)=\ker(\beta\epsilon)= I^i\cap \big((R\otimes_CB)\otimes_{\Z[P]\otimes_{\Z[Q]}\Z[N]}J\big)$, identifying $(R\otimes_CB)\otimes_{\Z[P]\otimes_{\Z[Q]}\Z[N]}J$ with its image in $I$ via $\alpha$, obtaining an exact sequence by definition of $D^i$
		\begin{equation}\label{ssstar}
		0 \longrightarrow I^i\cap \big((R\otimes_CB)\otimes_{\Z[P]\otimes_{\Z[Q]}\Z[N]}J\big) \longrightarrow I^i \longrightarrow D^i \longrightarrow 0 .
		\end{equation}
		
		Applying the Ker-Coker Lemma to the diagram of exact sequences for $1<j<i$,
		\[
		\begin{tikzcd}[column sep=3em,row sep=3em]
		0 \arrowr & I^i\cap \big((R\otimes_CB)\otimes_{\Z[P]\otimes_{\Z[Q]}\Z[N]}J\big) \arrowr\arrowd & I^i \arrowr\arrowd & D^i \arrow[tail]{d}\arrowr & 0 \\
		0 \arrowr & I^j\cap \big((R\otimes_CB)\otimes_{\Z[P]\otimes_{\Z[Q]}\Z[N]}J\big) \arrowr & I^j \arrowr & D^j \arrowr & 0
		\end{tikzcd}
		\]
		we obtain an exact sequence
		\[ 0 \longrightarrow \dfrac{I^j\cap \big((R\otimes_CB)\otimes_{\Z[P]\otimes_{\Z[Q]}\Z[N]}J\big)}{I^i\cap \big((R\otimes_CB)\otimes_{\Z[P]\otimes_{\Z[Q]}\Z[N]}J\big)} \longrightarrow I^j/I^i \longrightarrow D^j/D^i \longrightarrow 0 \]
		
		From the exact sequence \eqref{ssstar} we obtain
		\begin{align*}
			D^i&= I^i/\big( I^i\cap ((R\otimes_CB)\otimes_{\Z[P]\otimes_{\Z[Q]}\Z[N]}J)\big) \\
			&= \dfrac{I^i + \big((R\otimes_CB)\otimes_{\Z[P]\otimes_{\Z[Q]}\Z[N]}J\big)}{\big((R\otimes_CB)\otimes_{\Z[P]\otimes_{\Z[Q]}\Z[N]}J\big)}
		\end{align*}
		which is the ideal $\tilde{I}^i$ of $(R\otimes_CB) / \big((R\otimes_CB)\otimes_{\Z[P]\otimes_{\Z[Q]}\Z[N]}J\big) = B[X]$ where $\tilde{I}$ is the ideal of the simplicial ring $B[X]$ generated by the variables $X$. Therefore the hypotheses of \cite[Theorem 8.8]{Quillen-MIT} are verified (since $C\to B$ and $P\to N$ are surjective) and we obtain $H_n(D^i)=H_n(\tilde{I}^i)=0$ for all $i>n$.
		
		Putting $D^1:=D_{(B,N)|(C,Q)}^1$, we obtain a convergent spectral sequence \cite[XV, Proposition~4.1]{CE}
		\[ E_{p,q}^2=H_{p+q}\big(D^q/D^{q+1}\big) \;\Rightarrow\; H_{p+q}\big(D^1\big) . \]
		
		It remains to show that $D^1/D^2= \LLL_{(B,N)|(C,Q)}$. If we consider $R\otimes_CR$ as $R$-module via $r(a\otimes b)=a\otimes br$, we have a split exact sequence of $R$-modules (defining $I_\omega$)
		\[ 0 \longrightarrow I_\omega \longrightarrow R\otimes_CR \longrightarrow R \longrightarrow 0 .\]
		Applying $\ -\otimes_RB$, we obtain an split exact sequence
		\[ 0 \longrightarrow I_\omega\otimes_R B \longrightarrow R\otimes_CB \longrightarrow B \longrightarrow 0 \]
		which proves that $I_\omega\otimes_RB= I$.		
		
		Similarly, considering the ideal $J_\omega=\ker(\Z[P]\otimes_{\Z[Q]}\Z[P]\to \Z[P])$, we obtain $J_\omega\otimes_{\Z[P]}\Z[N]= J$. Therefore
		\[ I^i/ I^{i+1} = (I_\omega\otimes_R B)^i / (I_\omega \otimes_R B)^{i+1} = I_\omega^i\otimes_R B / I_\omega^{i+1} \otimes_R B = I_\omega^i/ I_\omega^{i+1} \otimes_R B\]
		and similarly
		\[ J^i/ J^{i+1} = J_\omega^i/ J_\omega^{i+1} \otimes_{\Z[P]} \Z[N] .\]
		
		Then the pushout \eqref{sstar} for $i=2$ takes the form
		\[
		\begin{tikzcd}[column sep=3em,row sep=4em]
		(R\otimes_CB)\otimes_{\Z[P]\otimes_{\Z[Q]}\Z[N]} (J_\omega/J_\omega^2\otimes_{\Z[P]}\Z[N]) \arrowr{\gamma}\arrowd{\lambda} & B\otimes_\Z T \arrowd \\
		B\otimes_RI_\omega/ I_\omega^2  \arrowr & D_{(B,N)|(C,Q)}^1/D^2
		\end{tikzcd}
		\]
		Since the homomorphisms $\lambda$ and $\gamma$ can be factorized by the surjective homomorphism
		\[ (R\otimes_CB)\otimes_{\Z[P]\otimes_{\Z[Q]}\Z[N]} \big(J_\omega/J_\omega^2\otimes_{\Z[P]}\Z[N]\big) \to B\otimes_{\Z[P]\otimes_{\Z[Q]}\Z[N]} \big(J_\omega/J_\omega^2\otimes_{\Z[P]}\Z[N]\big) \]
		that is obtained by applying $\ -\otimes_{\Z[P]\otimes_{\Z[Q]}\Z[N]} \big(J_\omega/J_\omega^2\otimes_{\Z[P]}\Z[N]\big)$ to the surjective morphism $R\otimes_CB\to B$, we obtain a pushout
		\[
		\begin{tikzcd}[column sep=3em,row sep=4em]
		B\otimes_{\Z[P]\otimes_{\Z[Q]}\Z[N]} (J_\omega/J_\omega^2\otimes_{\Z[P]}\Z[N]) \arrowr\arrowd & B\otimes_\Z T \arrowd \\
		B\otimes_RI_\omega/ I_\omega^2  \arrowr & D_{(B,N)|(C,Q)}^1/D^2
		\end{tikzcd}
		\]
		and additionally, since
		\begin{align*}
		B\otimes_{\Z[P]\otimes_{\Z[Q]}\Z[N]} (J_\omega/J_\omega^2\otimes_{\Z[P]}\Z[N]) &= B\otimes_{\Z[N]}\Z[N]\otimes_{\Z[P]\otimes_{\Z[Q]}\Z[P]} J_\omega/J_\omega^2 \\
		&= B\otimes_{\Z[N]} J_\omega/J_\omega^2
		\end{align*}
		(because the $\Z[P]\otimes_{\Z[Q]}\Z[P]$-module $J_\omega/J_\omega^2$ is already a $\Z[N]$-module),
		we obtain a pushout
		\[
		\begin{tikzcd}[column sep=3em,row sep=4em]
		B\otimes_{\Z[N]} (J_\omega/J_\omega^2) \arrowr\arrowd & B\otimes_\Z T \arrowd \\
		B\otimes_RI_\omega/ I_\omega^2  \arrowr & D_{(B,N)|(C,Q)}^1/D^2
		\end{tikzcd}
		\]
		
		Finally, we deduce $T=\ker\big(P^\gp/\ker(Q^\gp\to N^\gp)\to N^\gp\big)= P^\gp/Q^\gp $  given that $P_0^\gp= Q^\gp$. Thus, the previous pushout gives us $\LLL_{(B,N)|(C,Q)}$ by definition.
	\end{proof}

\end{document}